\documentclass[12pt]{article}
\usepackage{amsfonts}
\usepackage{amsmath}

\def\to{\rightarrow}

\def\EE{\mathbb{E} }
\def\PP{\mathbb{P} }
\def\ZZ{\mathbb{Z}}

\def\G{\mathcal{G}}
\def\GG{\mathbb{G}}

\def\RR{\mathbb{R} }
\newtheorem{theorem}{Theorem}

\newtheorem{corollary}[theorem]{Corollary}

\newtheorem{lemma}[theorem]{Lemma}

\newenvironment{proof}[1][Proof]{\noindent\textbf{#1.} }{\ \rule{0.5em}{0.5em}}
\begin{document}
\title{Weak Convergence of Stationary Empirical Processes}
\author{Dragan Radulovi\'c \\
{\small Department of Mathematics, Florida Atlantic University}\\
Marten Wegkamp\\
{\small Department of Mathematics \& Department of Statistical Science,
Cornell University}}
\date{\today}
\maketitle
 
\begin{abstract}
\noindent
We offer an umbrella type result  which extends   weak convergence
of classical empirical processes on the line to that of more general processes indexed by
functions of bounded variation. This extension is not contingent on the type
of dependence of the underlying sequence of random variables.
 As a consequence we establish  weak convergence for
stationary empirical processes indexed by general classes of functions under
$\alpha$-mixing conditions.\\

\noindent {\bf Running title}: Weak convergence of stationary empirical processes
\newline
\noindent {\bf MSC2000 Subject classification}: {Primary 60F17; secondary 60G99.}
\end{abstract}
\newpage

\section{\protect\bigskip Introduction}

We consider the empirical process 
\begin{equation}
\bar {\mathbb{Z}}_{n}(g) = \frac{1}{\sqrt{n}}
\sum_{i=1}^{n} \left\{ g(X_{i})-\mathbb{E} g(X_{i})\right\},\ {g\in \mathcal{G}},  \label{Zg}
\end{equation}%
indexed by  some  class  $\mathcal{G}$ of functions $g:\RR\to\RR$. It is 
 an obvious generalization of the classical process 
\begin{equation}
\mathbb{G}_{n}(t)=\frac{1}{\sqrt{n}} \sum_{i=1}^{n} \left\{ 1_{(-\infty ,t]}(X_{i})-F(t)\right\},
\label{Gt}
\end{equation}
in which case the indexing class is $\mathcal{G}=\left\{ g(x)=1_{(-\infty ,t]}(x)
:\ t\in \mathbb{R}\right\}$.
 If the underlying sequence $\{X_{i}\}_{i\ge1}$ is i.i.d.,
then the limiting behavior (as $n\to\infty$) of the empirical processes $\{\bar{\mathbb{Z}}_{n}(g),\ g\in\G\}$ is
well  understood. The same
is true for the bootstrap counterpart $\{\bar{\mathbb{Z}}_{n}^{\ast }(g),\ g\in\G\}$ based on an
i.i.d. bootstrap sample $\{X_{1,n}^{\ast },\ldots,X_{m_n,n}^\ast\}$, see, for instance, Van der Vaart and Wellner
(1996) and Dudley (1999). The theory of weak convergence of  empirical processes based on  independent
sequences has yielded a wealth of statistical applications and, in
particular, it was  instrumental for establishing the weak convergence
of numerous novel statistics. Often the limiting distributions of these
statistics do not allow for closed form solutions, in which case the
bootstrap version of the process  is
utilized. As is standard practice nowadays, see the thorough exposition in the first chapter of Van der Vaart and Wellner (1996),
weak convergence in this paper should be understood in the sense
of Hoffmann-J{\o}rgensen and expectations should be interpreted as outer expectations.\\

For empirical processes based on  stationary sequences $\{X_{i}\}_{i\ge1}$ the situation is rather different. The
classical process $\{\mathbb{G}_{n}(t),\, t\in\RR\}$ has been treated by numerous authors, who
established weak convergence under sharp
mixing assumptions, see, for instance, Rio (2000). However, nothing similar exists for more general
processes. The only work that we could find in the literature, Andrews
and Pollard (1994), treats more general indexing classes, but imposes rather
restrictive assumptions on the decay of $\alpha$-mixing coefficients. This
discrepancy between the conditions needed for $\bar{\mathbb{Z}}_{n}(g)$ and those for $
\mathbb{G}_{n}(t)$, is due to fact that the typical approach for proving the
uniform limiting theorems heavily relies on the estimation of entropy
numbers, which in turn require good exponential maximal inequalities. Only $\beta$-mixing, via decoupling, allows for such an estimate. This is the reason why one can find
in the literature  the treatment of $\bar{\mathbb{Z}}_{n}(g)$ only for $
\beta$-mixing sequences, see Arcones and Yu (1994) and Doukhan, Massart and Rio (1995).

The situation for   bootstrap of stationary empirical processes is even worse. Although  introduced more than twenty years ago,  we  only found three papers that
study the  bootstrap for stationary empirical processes indexed by general classes $\mathcal{G}$, see K\"{u}nsch (1989), Liu and Singh (1992)
and   Sengupta,  Volgushev and Shao  (2016). All these results operate under $\beta$-mixing assumptions. Moreover, only in the case of VC-classes do we have the
sharp conditions, see Radulovi\'c (1996). Bracketing classes were considered by
B\"uhlmann (1995), but this was established under   restrictive assumptions
on both $\beta$-mixing coefficients and   bracketing numbers. It is worth
mentioning that the 
covariance function  of the 
limiting Gaussian process is unknown, and consequently in most actual applications of
these results, we  heavily rely on the   bootstrap version of the
process  for which 
adequate results are sorely lacking. In short, the most general $\alpha$-mixing sequences
have never been considered for stationary bootstrap processes $\{\bar{\mathbb{Z}}_{n}^{\ast }(g),\ g\in\G\}$, while for the non-bootstrap version $\{
\bar{\mathbb{Z}}_{n}(g),\ g\in\G\}$ we
have only one example in the literature.\\

In what follows we prove two general results that allow us to extend  
weak convergence of   $\{\mathbb{G}_{n}(t),\ t\in\RR\}$ to that of $
\{\bar{\mathbb{Z}}_{n}(g), {g\in \mathcal{G}}\}$, where $\mathcal{G}$ is a class of
functions of uniformly bounded total variation (called $BV_{T}$ in this paper). We
would like to point out that although there are examples of Donsker classes
of infinite variation (for instance, $f:[0,1]\rightarrow \lbrack 0,1]$ with $
|f(x)-f(y)|\leq |x-y|^{\alpha }$, $1/2<\alpha <1)$, such cases are rather
the exception than the norm. The majority of examples of bounded Donsker
classes that are given in the literature are subsets of the class $BV_{T}$.\\

This enlargement from $\{\mathbb{G}_{n}(t), {t\in \RR}\}$ to $\{\bar{\mathbb{Z}}_{n}(g),\ {g\in \mathcal{G}}\}$ is not contingent on the dependence structure of
underlying sequence $\{X_{i}\}_{i\ge 1}$ (or $\{X_{i}^{\ast }\}_{i\ge 1}$) and the only
requirement is that $\mathbb{G}_{n}(t)$ converges weakly to a Gaussian
process. This allows us to derive  weak convergence  of $\{ \bar{\mathbb{Z}}_{n}(g),\ {g\in \mathcal{G}}\}$ and $\{ \bar{\mathbb{Z}}_{n}^{\ast }(g),\ {g\in \mathcal{G}
}\}$ for $\alpha$-mixing sequences. The same extension applies for the short
memory causal processes considered in Doukhan and Surgailis (1998), as well
as processes treated in Dehling, Durieu and Volny  (2009). The technique we employ is a  
simple application of the integration-by-parts formula, followed by  the continuous mapping theorem. Arguably this approach
may have been known before, although we   could not find any  communication of it in the literature, and   certainly not among the research
related to stationary empirical processes.
A follow-up paper (Radulovi\'c, Wegkamp and Zhao (2017)) extends the results obtained in this work  to classes indexed by {\em multivariate} functions of bounded variation, with an emphasis on empirical copula processes. 
An important technical difference with the follow-up paper is that here  we allow for general stationary distribution functions of the underlying $\{ X_i\}_{i\ge1}$.
The  proof for general, non-continuous processes, is not trivial caused by  technical complications related to the
interplay between the atoms of the limiting process $\ZZ(t)$ and
discontinuities of $g\in \mathcal{G}$.\\
 
The paper is organized as follows. Section 2 contains the statement of the
main result (Theorem 1) and related discussions, while Section 3 presents the  bootstrap version (Theorem 6) of this  result.
The proofs of these two main results are collected in Section 4. For completeness, the well-known
integration by parts formula can be found in the appendix.\\

\section{Main results}

For the total variation norm of a function $g:\mathbb{R}\rightarrow \mathbb{R
}$, we use the notation 
\begin{equation*}
||g||_{TV}=\sup_{\Pi }\sum_{x_{i}, x_{i+1} \in \Pi }|g(x_{i+1})-g(x_{i})|.
\end{equation*}
 The supremum is taken over all countable partitions $\Pi
=\{x_{1}<x_{2}<\ldots \}$ of $\mathbb{R}$. We set, for $T>0$, 
\begin{equation*}
BV_{T}:=\{g:\mathbb{R}\rightarrow \mathbb{R}:\ \Vert g\Vert _{TV}\leq T,
\text{ }\Vert g\Vert _{\infty }\leq T\}.
\end{equation*}
We let $BV_{T}^{\prime }\subset BV_{T}$ be the class of all
 functions $g$ in $BV_T$ that are  right-continuous.
Finally, we let $\{\mathbb{Z}_{n}(t),\ t\in \mathbb{R}\}$ be
an arbitrary stochastic process such that
\begin{itemize}
\item[A1:]
$\lim_{|t|\rightarrow \infty }\mathbb{Z}_{n}(t)=0$;
\item[A2:] \textit{The sample paths of }$\mathbb{Z}_{n}$\textit{\ are
right-continuous and of bounded variation}.
\end{itemize}
Clearly, both requirements A1 and A2 are met for the canonical empirical
process $\{\mathbb{G}_{n}(t),\ t\in\RR\}$. In this paper, we study the limit distribution (as $n\to\infty$) 
of the sequence of processes 
\begin{equation*}
\left\{\, \overline{\mathbb{Z}}_{n}(g):=\int g(x)\,\mathrm{d}\mathbb{Z}_{n}(x),\ \
g\in \mathcal{G}\, \right\}
\end{equation*}%
for some class $\mathcal{G}\subseteq BV_{T}^{\prime }$, for some finite $T$.\\

Although the motivation as well as the most notable applications of our results are related to the canonical case $\mathbb{Z}_{n}(x)=\mathbb{G}_{n}(x)$,
 the actual proof carries over for any process $\mathbb{Z}_{n}$
as long as  assumptions A1 and A2 are satisfied. 

Our main result is the following theorem.

\begin{theorem}\label{theorem 4}
Assume that   $\{\ZZ_{n}(t),\ t\in\RR\}$
converges weakly, as $n\to\infty$, to a  Gaussian process $\{\ZZ(t), \ t\in\RR\}$, that has
uniformly continuous sample paths with respect to the distance $d(s,t)=|F_0(s)-F_0(t)|$ for some distribution function $F_0$.\\
Then,  for any $T<\infty$ and $\G\subseteq BV_T'$,  $\{ \bar{\ZZ}_{n}(g), \ g\in\G\}$   
converges weakly, as $n\to\infty$, to a 
Gaussian process in $\ell^\infty(\G)$, that has uniformly $L_1(F_0)$-continuous sample paths.
\end{theorem}

\paragraph{Remark.}
Usually, an envelope condition $\sup |g(x)|\leq G(x)$ and $G\in L_2(F_0)$ is imposed. Since such a condition, coupled with
$\Vert g\Vert_{TV}\leq T$, implies that the functions $g$ are uniformly bounded, we 
assume $\Vert g\Vert _{\infty }\leq T$ in our definition of $BV_T$. 
\\

Theorem 1 allows us to derive weak convergence of $\bar{\mathbb{Z}}_{n}(g)$ via $\mathbb{Z}_{n}(t)$, regardless of the structure of the latter process. For
instance, taking $\mathbb{Z}_{n}$ as the standard empirical processes $\mathbb{G}_{n}$ based on stationary sequences $\{X_{i}\}_{i\ge 1}$, we obtain the
following corollary as an immediate consequence of Theorem 1.

\begin{corollary}
Let $\{X_k\}_{k\ge 1}$ be a stationary sequence of  random
variables    with  distribution $F$ and $\alpha$-mixing coefficients $\alpha_n$ satisfying $\alpha
_{n}= O(n^{-r}) $, for some $r>1$ and all  $n\ge1$.
 Then,  for any    $\G\subseteq BV_T'$, 
 \begin{itemize}
 \item[-]
  $ \{\int g\, {\rm d} \GG_n$, $g\in\G\}$ converges weakly to a 
 Gaussian process with uniformly  $L_1(F)$-continuous sample paths in $\ell^\infty(\G)$.
\item[-]
For continuous $F$,   $\mathcal{G}$  can be enlarged to $BV_{T}$.
\end{itemize}
\end{corollary}
\begin{proof}
 It is well known, see Theorem 7.2, page
96 in Rio (2000), that $\alpha _{n}=O(n^{-r})$ with $r>1$,  implies that the standard
empirical process $\mathbb{G}_{n}(t)$ converges weakly to a Brownian bridge
process with uniformly continuous paths with respect to the distance $
d(s,t)=|F(s)-F(t)|$, for the stationary distribution $F$ of $X_{k}$. The
result for $\mathcal{G}\subseteq BV_{T}^{\prime }\,\ $now follows trivially
from Theorem 1. If   $F$ is continuous, then the limiting
Brownian bridge has uniformly continuous sample paths with respect to the  Lebesgue measure
on $\mathbb{R}$. This, combined with   Lemma 3 below, 
implies Corollary 2.
\end{proof}
\bigskip

We used the following lemma.

\begin{lemma}
We have, for $\bar{\ZZ}_n(g)$ based on the canonical process $\ZZ_n(x)=\GG_n(x)$,
\begin{equation*}
\sup_{g\in BV_{T}}\inf_{h\in BV_{T}^{\prime }}\left\vert \bar{\mathbb{Z}}
_{n}(g)-\bar{\mathbb{Z}}_{n}(h)\right\vert \leq T\sup_{x}|\mathbb{G}_{n}(x)-
\mathbb{G}_{n}(x^{-})|.
\end{equation*}
\end{lemma}
 \begin{proof} Let $g$ be an arbitrary function in $BV_{T}.$ We denote its   countable many
discontinuities  by $a_{i}.$
Let $\overline{g}$ be the right-continuous version of $g$, that is, $
\overline{g}(x)=g(x)$ for all $x\neq a_{i}$ and $\overline{g}
(a_{i})=g(a_{i}^{+})$ for all $i$. Then
\begin{eqnarray*}
\left\vert \int g\, {\rm d} \mathbb{G}_{n}-\int \overline{g}\, {\rm d} \mathbb{G}_{n}\right\vert
&\leq& \sum_{i}|g(a_{i})-\overline{g}(a_{i})||\mathbb{G}_{n}(a_{i})-\mathbb{G}
_{n}(a_{i}^{-})|\\
&\leq& \left\Vert g\right\Vert _{TV}\sup_{x}|\mathbb{G}_{n}(x)-
\mathbb{G}_{n}(x^{-})|
\end{eqnarray*}
and the  conclusion  follows.
\end{proof}
\bigskip

We would like to point out that $\alpha$-mixing is the least restrictive form
of available mixing assumptions in the literature. To the best of our
knowledge, there are actually very few results that treat processes $\{ \int g\,\mathrm{d}\mathbb{G}_{n}, \, g\in \mathcal{G}\}$
indexed by functions, and they all require very stringent conditions on the
entropy numbers of $\mathcal{G}$ and on the rate of decay for $\alpha _{k}$.
See, for instance, Andrews and Pollard (1994). This is due to fact that
$\alpha$-mixing does not allow for sharp exponential inequalities for partial
sums. Consequently, the only known cases for which we have sharp conditions
are under more restrictive, $\beta$-mixing dependence. Indeed, $\beta$-mixing allows for
decoupling and it does yield exponential inequalities not unlike the i.i.d.
case. The current state-of-the art results, Arcones and Yu (1994), Doukhan,
Massart and Rio (1995), applied to bounded sequences, require $\sum_{n}\beta
_{n}<\infty $.\\

While it is correct that $\alpha$-mixing is the weakest of the known mixing concepts, a referee pointed out that functionals of mixing processes constitutes a much weaker concept of short range dependence.  For instance, Billingsley (1968, Theorem 22.2) establishes the empirical process CLT for certain functionals of $\phi$-mixing processes.
The same  referee made us aware  that Theorem 1 also applies to the empirical process of long-range dependent data, established by Dehling and Taqqu (1989) for subordinate Gaussian processes, and by Ho and Hsing (1996) for linear processes.
We provide an example  to demonstrate that Theorem 1 goes  beyond dependence defined via mixing conditions. It applies to short memory causal linear sequences $\{ X_i\}_{i\ge1}$ that
are defined by 
$
X_{i}=\sum_{j=0}^{\infty }a_{j}\xi _{i-j}
$
based on i.i.d. random variables $\xi _{i}$ and constants $a_{i}$. While the 
$X_{i}$ form a stationary sequence, they do not necessarily satisfy any
mixing condition. Weak convergence of the empirical processes $\{\mathbb{G}_{n}(t), \ t\in \mathbb{R}\}$ was established under sharp conditions in
Doukhan and Surgailis (1998) on the weights $a_j$ ($\sum_{j\ge 0} |a_j|^\gamma<\infty$ for some $0< \gamma\le 1$),
characteristic function of $\xi_0$ ($|\EE[ \exp( i u \xi_0 ) ] | \le C (1+|u|)^{-p}$ for all $u\in\RR$, some $C<\infty$ and $2/3<  p \le 1$),
 and a moment condition on $\xi_0$
($\EE[ |\xi_0|^{4\gamma} ] < \infty$).
 To the best of our knowledge, there are no
extensions to the more general processes $\{\bar{\mathbb{Z}}_{n}(g),\, g\in\G\}$. Theorem 1 and
the Doukhan and Surgailis (1998) result combined imply the following:\\


\begin{corollary}
Let 
$\{X_{i}=\sum_{j\geq 0}a_{j}\xi _{i-j}\}_{i\ge1} $
 be such that conditions of Doukhan and Surgailis (1998, pp 87--88)
are satisfied and let $F$  be the stationary distribution of $X_{i} $.
Then, for any $\mathcal{G}\subseteq BV_{T}$,  $\{\int g\, {\rm d} \mathbb{G}_{n},\ 
g\in \mathcal{G}\}$  converges weakly to a 
Gaussian process with uniformly $L_{1}(F)$-continuous sample paths in $\ell ^{\infty}(\mathcal{G})$.
\end{corollary}
\begin{proof}
 Doukhan and Surgailis (1998) proved that $
\{\mathbb{G}_{n}(t),\ t\in\RR\}$ converges weakly to a Gaussian process in the Skorohod
space. Since $F$ is continuous under their assumptions, see Doukhan and
Surgailis (1998, pp 88), the sample paths of limiting process of $\mathbb{G}_{n}$ are
uniformly continuous with respect to the distance $d(s,t)=|F(s)-F(t)|$ based on  the
stationary distribution $F$. The proof now follows trivially from Theorem 1
and Lemma 3.
\end{proof}
\bigskip

The recent papers by Dehling, Durieu and Volny (2009)  and Dehling,  Durieu and  Tusche (2014) offer yet another,
clever way to prove the weak limit of the standard empirical processes $\mathbb{G}_{n}$ based on stationary sequences that are not necessarily
mixing. Their technique uses finite dimensional convergence coupled with a
bound on the higher moments of partial sums, which in turn controls the
dependence structure. Dehling, Durieu and Volny (2009) establishes the weak convergence
of $\mathbb{G}_{n}$, while Dehling,  Durieu and  Tusche (2014) extends this idea to $\bar{\ZZ}_n$ indexed by more
general classes of functions. However, these authors impose  cumbersome entropy
conditions and only manage to marginally extend the classes. For example,
they  prove weak convergence of the process $\int f_{t}\,\mathrm{d}%
\mathbb{G}_{n}$, indexed by functions $f_{t}(x)$ which constitute an
one-dimensional monotone class (under the restrictive requirement that $\ s\leq
t\Rightarrow f_{s}\leq f_{t}$). Theorem 1 applied in their setting, yields a
more general result.

\begin{corollary}
Let  $\mathcal{G}\subseteq BV_{T}$ and
let  $F$ be the stationary distribution of the sequence $\{X_{i}\}_{i\ge1}$. Under
assumptions (i) and (ii) in Section 1 of Dehling, Durieu and Volny (2009),   $\{\int g\,{\rm d}\mathbb{G}_{n},\ g\in \mathcal{G}\}$ converges weakly to a 
Gaussian process with uniformly $L_{1}(F)$-continuous sample paths in $\ell ^{\infty }(\mathcal{G})$.
\end{corollary}
\begin{proof}
The underlying distribution function $F$ of $X_{k}$ in Dehling, Durieu and Volny (2009) is continuous, see their display (3) at page 3702.
Moreover, these authors establish   weak convergence of $\mathbb{G}_{n} $ to a Gaussian process that has uniformly continuous sample paths with respect to the distance $d(s,t)=|F(s)-F(t)|$. The proof  follows trivially from Theorem 1 and Lemma 3.
\end{proof}

\section{Bootstrap}

The weak limit of $\bar{\mathbb{Z}}_{n}$ based on $\GG_n$ in Theorem 1 is a Gaussian process $\bar{\ZZ}$ with  
complicated covariance structure 
\begin{eqnarray*}
\mathbb{E}[\bar{\mathbb{Z}}(f)\bar{\mathbb{Z}}(g)] &=&\text{\textrm{Cov}}
(f(X_{0}),g(X_{0}))+\sum_{k=1}^{\infty }\text{\textrm{Cov}}
(f(X_{0}),g(X_{k})) \\
&&+\sum_{k=1}^{\infty }\text{\textrm{Cov}}(f(X_{k}),g(X_{0}))
\end{eqnarray*}
for $f,g\in \G\subset BV_T$.
A   closed form solution is  seldom available, so that   actual applications of   weak limit results   of $\bar{\mathbb{Z}}_{n}$ are   hard to
implement. This situation calls for the bootstrap principle. Given
the sample of first $n$ observations $\{X_{1},\ldots ,X_{n}\}$, we let $\mathbb{G}_{n}^{\ast }$ be the
bootstrap empirical process 
\begin{equation*}
\mathbb{G}_{n}^{\ast }(t)=\sqrt{m_{n}}\left( \frac{1}{m_{n}}
\sum_{i=1}^{m_{n}}1_{(-\infty ,t]}(X_{i}^{\ast })-\frac{1}{n}
\sum_{i=1}^{n}1_{(-\infty ,t]}(X_{i})\right) ,\ t\in \mathbb{R},
\end{equation*}
based on a bootstrap sample $\{X_{1,n}^{\ast },\ldots ,X_{m_{n},n}^{\ast }\}$. 
We
stress that  no additional assumption on the structure
of the variables $X_{i,n}^{\ast }$ is required for our purposes. Analogous to $\bar{\mathbb{Z}}_{n}(g)=\int g\,\mathrm{d}\mathbb{G}_{n}$, we define $\bar{\mathbb{Z}}_{n}^{\ast
}(g):=\int g\,\mathrm{d}\mathbb{G}_{n}^{\ast }$ for any $g\in BV_{T}^{\prime
}$ with $T<\infty $. We recall that the bounded Lipschitz distance 
\begin{equation*}
d_{BL}(\mathbb{Z}_{n}^{\ast },\mathbb{Z})=\sup_{h\in BL_{1}}\left\vert 
\mathbb{E}^{\ast }[h(\mathbb{Z}_{n}^{\ast })]-\mathbb{E}[h(\mathbb{Z}%
)]\right\vert
\end{equation*}%
between two processes $\mathbb{Z}_{n}^{\ast }$ and $\mathbb{Z}$ metrizes
weak convergence. The symbol $\mathbb{E}^{\ast }$ stands for  the expectation over the
randomness of the bootstrap sample $X_{1,n}^{\ast },\ldots ,X_{m_{n},n}^{\ast }$, conditionally given the original sample $X_{1},\ldots ,X_{n}$,
while  $BL_{1}$ denotes the
 space of Lipschitz functionals $h:\ell ^{\infty }(\mathcal{G}%
)\rightarrow \mathbb{R}$ with $|h(x)|\leq 1$ and $|h(x)-h(y)|\leq \Vert
x-y\Vert _{\infty }$ for all $x,y\in \mathcal{G}\subset BV_{T}$. As is customary
in the literature, we speak of weak convergence
in probability if the random variable $d_{BL}(\mathbb{Z}_{n}^{\ast },\mathbb{Z})$ converges to zero in probability; if it converges to zero almost surely, we speak of weak
convergence almost surely.

\bigskip

\begin{theorem}
Let $\mathcal{G}\subseteq BV_{T}^{\prime }$.
Assume that, conditionally on  $X_{1},\ldots ,X_{n}$, 
 $\{\GG_{n}^{\ast }(t),\ t\in \RR\}$  converges weakly to a
Gaussian process, in probability, that has unifomly continuous sample paths with respect to the distance  $
d(s,t)=|F(s)-F(t)|$ based on the stationary distribution $F$ of  $X_{i}$.
The following three statements hold true:
\begin{enumerate}
\item
  $\left\{   \int g\, {\rm d} \mathbb{G}_{n}^{\ast },\ g\in \mathcal{G} \right\}$  converges
weakly to a Gaussian process with   uniformly  $L_{1}(F)$-continuous sample paths in $\ell
^{\infty }(\mathcal{G})$.
\item  If the weak convergence of $\{ \mathbb{G}_{n}^{\ast } (t),\ t\in\RR\}$    holds almost surely, then the conclusion that  $\left\{   \int g\, {\rm d} \mathbb{G}_{n}^{\ast },\ g\in \mathcal{G} \right\}$  converges weakly in $\ell ^{\infty }(\G)$ holds almost surely as well.
\item
If  $\{\mathbb{G}_{n}(t),\ t\in\RR\} $ and $\{\mathbb{G}_{n}^{\ast }(t),\ t\in\RR\}$  converge to the same limit, then so do  $\left\{   \int g\, {\rm d} \mathbb{G}_{n},\ g\in \mathcal{G} \right\}$
and  $\left\{   \int g\, {\rm d} \mathbb{G}_{n}^{\ast },\ g\in \mathcal{G} \right\}$.
\end{enumerate}
\end{theorem}

The literature offers numerous
bootstrapping techniques for stationary data, such as  moving block bootstrap,
stationary bootstrap, sieved bootstrap, Markov chain bootstrap, to name a few, but
their validity is proved for specific cases/statistics only. Due to
complications with entropy calculations for dependent triangular arrays,
almost all results treat the standard empirical processes $\mathbb{G}%
_{n}^{\ast }$ with few notable exceptions. The moving block bootstrap was
justified for VC-type classes, but only under rather restrictive $\beta$-mixing
conditions on $X_{i}$ (Radulovi\'{c}, 1996). Bracketing classes were
considered by B\"{u}hlmann (1995), but his conditions are even more
restrictive. 
In contrast, the process $\{\mathbb{G}_{n}(t),$ $t\in \mathbb{R}\}$ is
rather easy to bootstrap. This coupled with Theorem 6 offers the following
result. 

\begin{corollary}
Let $\{X_{j}\}_{j\ge1}$  be a stationary sequence
of random variables with continuous stationary distribution function $F$  and
$\alpha$-mixing coefficients satisfying $\sum_{k\geq n}\alpha
_{k}\le C n^{-\gamma }$, for some $0<\gamma <1/3$ and $C<\infty$. Let 
$\GG_{n}^{\ast }$
 be the bootstrapped standard empirical process based
on the moving block bootstrap, with block sizes  $b_{n}$, specified
in Peligrad (1998, page 882).
Then, for $\mathcal{G}\subset BV_{T}$,
$\{\int g\, {\rm d} \GG_{n}^{\ast },\
g\in \G\}$  converges weakly to
Gaussian process, almost surely, with uniformly $L_1(F)$-continuous sample paths. 
\end{corollary}
\begin{proof} Theorem 2.3 of Peligrad (1998) establishes
the convergence of $\mathbb{G}_{n}^{\ast }$ to a Gaussian process with uniformly
continuous sample paths with respect to the distance $d(s,t)=|F(s)-F(t)|$ based on the
stationary distribution $F$ of $X_{i}$.
 Invoke
Theorem 6 and Lemma 3 to conclude the proof.
\end{proof}

\bigskip

Just as for  weak convergence of the empirical process based on
stationary sequences, there are numerous results that consider  
bootstrap for stationary, non-mixing sequences. For example, El Ktaibi,   Ivanoff and Weber (2014) study
short memory causal linear sequences, and prove weak convergence of $\mathbb{G}_{n}^{\ast }$ under conditions (on the growth of the weights $a_j$, the characteristic  function of $\xi_0$ and moments of $\xi_0$)  akin to the ones required for its
non-bootstrap counterpart $\mathbb{G}_{n}$ (Doukhan and Surgailis, 1998).
Again, Theorem 6 easily extends their result.

\begin{corollary}
Let $\{X_{i}=\sum_{j\geq 0}a_{j}\xi _{i-j}\}_{i\ge1}$
be a sequence of random variables with stationary distribution $F$
such that conditions of El Ktaibi,   Ivanoff and Weber  (2014) are satisfied. Then,
for any $\mathcal{G}\subset BV_{T}$,  $\{\int g\, {\rm d} \GG_{n}^{\ast },\ g\in \mathcal{G}\}$ 
converges weakly to a 
Gaussian limit with uniformly $L_1(F)$-continuous sample paths, almost surely.
\end{corollary}

\begin{proof}
El Ktaibi,   Ivanoff and Weber (2014) establish weak
convergence of $\mathbb{G}_{n}^{\ast }$ in the Skorohod space, for
continuous $F,$ which in turn implies that the limiting process has
uniformly continuous sample paths with respect to distance $d(s,t)=|F(s)-F(t)|$. Again, invoke Theorem 6 and Lemma 3 to conclude the proof.
\end{proof}

\section{Proofs for Theorem 1 and Theorem 6}
We first give a short proof of Theorem 1 
if the limit $\mathbb{Z}(t)$ of $\mathbb{Z}_{n}(t)$ has  continuous sample paths with respect to the  Lebesgue measure. 
Let $d_{BL}$ be the bounded Lipschitz metric that metrizes  weak
convergence, see, e.g., Van der Vaart and  Wellner (1996, page 73) for the
definition. 
Set $\bar{\mathbb{Z}}_{n}(g)=\int
g\,\mathrm{d}\mathbb{Z}_{n}$ for any $g\in BV_{T}^{\prime }$. Assumptions A1 and A2 imply that the Lebesgue-Stieltjes integrals 
\begin{equation*}
\widetilde{\mathbb{Z}}_{n}(g)=-\int \mathbb{Z}_{n}\,\mathrm{d}g
\end{equation*}%
are well defined. Next, by the integration by parts formula in Lemma A, we
have 
\begin{equation*}
\bar{\mathbb{Z}}_{n}(g)=\tilde{\mathbb{Z}}_{n}(g)+R_{n}(g)
\end{equation*}%
with 
\begin{equation*}
R_{n}(g)\leq T\sup_{x}|\mathbb{Z}_{n}(x)-\mathbb{Z}_{n}(x^{-})|\rightarrow 0,
\end{equation*}%
in probability, as $n\rightarrow \infty $, since $\mathbb{Z}$ has uniformly continuous sample paths.
For fixed $g$, $\tilde{\mathbb{Z}}_{n}(g)$ converges weakly to $\tilde{%
\mathbb{Z}}(g):=-\int \mathbb{Z}\,\mathrm{d}g$ by the continuous mapping
theorem and weak convergence of $\mathbb{Z}_{n}$. The continuous mapping
theorem also guarantees that the limit $\tilde{\mathbb{Z}}$ is tight in $%
\ell ^{\infty }(\mathcal{G})$ as the map 
$X\mapsto \Gamma _{g}(X)=-\int
X\,\mathrm{d}g$, $g\in \mathcal{G}$, is continuous. By the triangle
inequality, 
\begin{eqnarray*}
d_{BL}(\bar{\mathbb{Z}}_{n},\tilde{\mathbb{Z}}) &\leq &d_{BL}(\bar{\mathbb{Z}%
}_{n},\tilde{\mathbb{Z}}_{n})+d_{BL}(\tilde{\mathbb{Z}}_{n},\tilde{\mathbb{Z}%
}) \\
&=&\sup_{H\in BL_{1}}|\mathbb{E}H(\bar{\mathbb{Z}}_{n})-\mathbb{E}H(\tilde{%
\mathbb{Z}}_{n})|+\sup_{H\in BL_{1}}|\mathbb{E}H(\tilde{\mathbb{Z}}_{n})-%
\mathbb{E}H(\tilde{\mathbb{Z}})| \\
&\leq &T\mathbb{E}[\sup_{x}|\mathbb{Z}_{n}(x)-\mathbb{Z}_{n}(x^{-})|]+T%
\sup_{H^{\prime }\in BL_{1}}|\mathbb{E}H^{\prime }(\mathbb{Z}_{n})-\mathbb{E}%
H^{\prime }(\mathbb{Z})| \\
&=&T\mathbb{E}[\sup_{x}|\mathbb{Z}_{n}(x)-\mathbb{Z}_{n}(x^{-})|]+Td_{BL}(%
\mathbb{Z}_{n},\mathbb{Z})
\end{eqnarray*}%
The second inequality follows since the map $\Gamma _{f}(X):=\int X\ \mathrm{d}f$
is linear and Lipschitz with Lipschitz constant $\int \,|\mathrm{d}f|\leq T$ and the
suprema are taken over all Lipschitz functionals $H:\ell ^{\infty }(\mathcal{%
G})\rightarrow \mathbb{R}$ with $\Vert H\Vert _{\infty }\leq 1$ and $%
|H(X)-H(Y)|\leq \Vert X-Y\Vert _{\infty }$ and $H^{\prime }:\ell ^{\infty }(%
\mathbb{R})\rightarrow \mathbb{R}$ with $\Vert H^{\prime }\Vert _{\infty
}\leq 1$ and $|H^{\prime }(X)-H^{\prime }(Y)|\leq \Vert X-Y\Vert _{\infty }$%
, respectively. Together with the tightness of the limit $\tilde{\mathbb{Z}}$%
, this implies that the empirical process $\bar{\mathbb{Z}}_{n}(g)$ indexed
by $g\in \mathcal{G}\subset BV_{T}'$ converges weakly.

\bigskip

\bigskip The above proof for continuous limit processes $\ZZ$ (that is, processes $\ZZ$ with uniformly continuous sample paths)  is
rather simple.  
Nevertheless, we could not find an
actual publication of this trick. We
would like to stress that the extension of this proof to general, non-continuous processes, is not trivial. Technical complications related to the
interplay between the atoms of the limiting process $\ZZ(t)$ and
discontinuities of $g\in \mathcal{G},$ require some care.

\begin{lemma}
Assume that  
$\mathbb{Z}_{n}(t)$
converges weakly to a Gaussian process 
$\mathbb{Z}(t)$, as $n\to\infty$.
Then, for any right continuous function $g:\mathbb{R}
\rightarrow \mathbb{R}$\textit{\ of bounded variation }$\overline{\mathbb{Z}}
_{n}(g)$\textit{=}$\int g\, {\rm d} \mathbb{Z}_{n}$\textit{\ is well defined, and
converges to a normal distribution on }$\mathbb{R}.$
\end{lemma}
\begin{proof} Let $g$ be an arbitrary right-continuous function
of bounded variation. First, we notice that $\int g\, {\rm d} \mathbb{Z}_{n}$ and $
\int \mathbb{Z}_{n}\, {\rm d} g$ are indeed well defined as Lebesgue-Stieltjes
integrals. Recall that $g$ can have only countably many discontinuities
which we will denote $a_{i}.$ By the integration by parts formula, Lemma A in
the appendix, we have
\begin{equation*}
\int g(x)\, {\rm d}\mathbb{Z}_{n}(x)=T_{1}(\mathbb{Z}_{n})+T_{2}(\mathbb{Z}_{n})
\end{equation*}
with operators $T_{1},T_{2}:\ell^{\infty }(\mathbb{R})\rightarrow \mathbb{R}$
\begin{eqnarray*}
T_{1}(\mathbb{Z}_{n})&:=&-\int \mathbb{Z}_{n}(x)\, {\rm d} g(x)
\\
T_{2}(\mathbb{Z}_{n})&:=&\int \int 1_{x=y}\, {\rm d} g(x)\, {\rm d}\mathbb{Z}_{n}(y)
=\sum_{i}(g(a_{i})-g(a_{i}^{-}))(\mathbb{Z}_{n}(a_{i})-\mathbb{Z}_{n}(a_{i}^{-})).
\end{eqnarray*}
Since $g$ has finite variation, it is bounded with   $\sum_{i}|\alpha _{i}|<\infty $  
for the jumps $\alpha_{i}:=g(a_{i})-g(a_{i}^{-})$.
 Hence, the operator $T_2$, given by
\begin{equation*}
T_{2}(\mathbb{Z}_{n})=\sum_{i}\alpha _{i}(\mathbb{Z}_{n}(a_{i})-\mathbb{Z}_{n}(a_{i}^{-}))
\end{equation*}
is linear and continuous. We conclude the proof by observing that the continuity of the
operators $T_{1}$ and $T_{2}$ and the weak convergence of $\mathbb{Z}_{n}(t)$
to a Gaussian process $\mathbb{Z}$ ensures that these sequences $Y_{n}:=\int
g \, {\rm d}\mathbb{Z}_{n}=T_{1}(\mathbb{Z}_{n})+T_{2}(\mathbb{Z}_{n})$ converge
weakly to a  normal distribution via the continuous mapping theorem.
\end{proof}
\bigskip

For any distribution function $F_{0}$ and any $\beta >0$, we define 
\begin{eqnarray*}
\Psi _{\beta ,F_{0}}(\mathbb{Z}_{n}):= &&\sup_{|F_{0}(s)-F_{0}(t)|\leq \beta
}\left\vert \mathbb{Z}_{n}(s)-\mathbb{Z}_{n}(t)\right\vert , \\
{\mathbb{g}}_{\beta ,F_{0}}(\mathbb{Z}_{n}):=
&&\sup_{F_{0}(s)-F_{0}(s^{-})>\beta }\left\vert \mathbb{Z}_{n}(s)-\mathbb{Z}_{n}(s^{-})\right\vert ,
\end{eqnarray*}
with the convention that the supremum taken over the empty set is equal to zero.

Clearly, ${\mathbb{g}}_{\beta ,F_{0}}(\mathbb{Z}_{n})=0$ for all $\beta >0$
if $F_{0}$ is continuous. In general, for arbitrary $F_{0}$, the quantity ${\mathbb{g}}_{\beta ,F_{0}}(\mathbb{Z}_{n})$ is bounded in probability, for
all $\beta >0$, by the continuous mapping theorem, as long as $\mathbb{Z}_{n} $ converges weakly. The following lemma  plays an instrumental role and it could be perhaps of  
independent interest.  

\begin{lemma}[Decoupling Lemma]
For any distribution function  $F_{0}$  on  $\mathbb{R}$, any right-continuous function 
$g:\mathbb{R}\rightarrow \mathbb{R}$ and any $\beta >0$, we
have
\begin{equation*}
|\bar{\mathbb{Z}}_{n}(g)|\leq \left\{ 2\beta
^{-1}||g||_{L_{1}(F_{0})}+6||g||_{TV}\right\} \Psi _{2\beta ,F_{0}}(\mathbb{Z}_{n})+\beta ^{-1}\Vert g\Vert _{L_{1}(F_{0})}{\mathbb{g}}_{\beta ,F_{0}}(
\mathbb{Z}_{n}).
\end{equation*}
Here $\Vert g\Vert _{L_{1}(F_{0})}=\int |g|\,\, {\rm d} F_{0}$
and $\mathbb{Z}_{n}$ satisfies assumptions A1 and A2.
\end{lemma}
\begin{proof}
 Without loss of generality we can assume
that $||g||_{L_{1}(F_{0})}+||g||_{TV}<\infty .$ Since $F_{0}$ is a
distribution function we can construct, for any $0<\beta <1,$ a finite grid $-\infty =s_{0}<s_{1}<...<s_{M}<\infty $ such that
\begin{equation*}
F_{0}(s_{j})-F_{0}(s_{j-1})\geq \beta ,\text{ }F_{0}(s_{M})<1-2\beta
\end{equation*}
and 
\begin{equation*}
F_{0}(s_{j}^{-})-F_{0}(s_{j-1})\leq 2\beta ,
\end{equation*}
leaving the possible jumps $F_{0}(s_{j})-F_{0}(s_{j}^{-})$ unspecified$.$
Based on this grid we approximate $\mathbb{Z}_{n}(t)$ by 
\begin{equation*}
\widetilde{\mathbb{Z}}_{n}(t):=\sum_{i=1}^{M}\mathbb{Z}_{n}(s_{j-1})1_{[s_{j-i,}s_{j})}(t)
\end{equation*}
and we set $\widetilde{\mathbb{Z}}_{n}(\pm \infty )=0.$ We observe that by
construction 
\begin{eqnarray}
\sup_{x}|\mathbb{Z}_{n}(x)-\widetilde{\mathbb{Z}}_{n}(x)| &\leq& \max_{1\leq
j\leq M}\sup_{x\in \lbrack s_{j-1},s_{j})}|\mathbb{Z}_{n}(x)-\mathbb{Z}_{n}(s_{j-1})|+\sup_{x\geq s_{M}}|\mathbb{Z}_{n}(x)|\nonumber\\
&\leq& \sup_{|F_{0}(x)-F_{0}(y)|\leq 2\beta }|\mathbb{Z}_{n}(x)-\mathbb{Z}_{n}(y)|=\Psi _{\beta ,F_{0}}(\mathbb{Z}_{n})  \label{SupB}
\end{eqnarray}
Since the process $\widetilde{\mathbb{Z}}_{n}$ inherits the bounded
variation properties of $\mathbb{Z}_{n},$
\begin{equation*}
\int g(s)\, {\rm d} \mathbb{Z}_{n}(s)=\int g(s)\, {\rm d} \widetilde{\mathbb{Z}}_{n}(s)+\int
g(s)\, {\rm d} (\mathbb{Z}_{n}-\widetilde{\mathbb{Z}}_{n})(s).
\end{equation*}
is well defined for any right-continuous function $g$ of bounded variation.
Using the integration by parts formula (Lemma A in the appendix) we obtain for the last term
on the right
\begin{equation*}
\int g(s)\, {\rm d} (\mathbb{Z}_{n}-\widetilde{\mathbb{Z}}_{n})(s)=-\int (\mathbb{Z}_{n}-\widetilde{\mathbb{Z}}_{n})(s)\, 
{\rm d} g(s)+\int \int 1_{x=y}\, {\rm d} g(x)\, {\rm d} (\mathbb{Z}_{n}-\widetilde{\mathbb{Z}}_{n})(y).
\end{equation*}
Since $g$ is of bounded variation, it has countably many discontinuities $a_{i}$.
 Using (\ref{SupB}), we obtain 
\begin{eqnarray*}
\left\vert \int \int 1_{x=y}\, {\rm d} g(x)\, {\rm d}(\mathbb{Z}_{n}-\widetilde{\mathbb{Z}}_{n})(y)\right\vert
&\leq& \sum_{i}|g(a_{i})-g(a_{i}^{-})|\left\vert (\mathbb{Z}_{n}-\widetilde{
\mathbb{Z}}_{n})(a_{i})-(\mathbb{Z}_{n}-\widetilde{\mathbb{Z}}_{n})(a_{i}^{-})\right\vert \\
&\leq& 2\left\Vert g\right\Vert _{TV}\Psi _{2\beta ,F_{0}}(\mathbb{Z}_{n}).
\end{eqnarray*}
Consequently,
\begin{equation*}
\left\vert \int g(s)\, {\rm d} (\mathbb{Z}_{n}-\widetilde{\mathbb{Z}}
_{n})(s)\right\vert \leq 3\left\Vert g\right\Vert _{TV}\Psi _{2\beta ,F_{0}}(
\mathbb{Z}_{n})
\end{equation*}
 Next we deal with the finite dimensional approximation 
\begin{equation*}
\int g(s)d\widetilde{\mathbb{Z}}_{n}(s)=\sum_{j=1}^{M}g(s_{j})(\mathbb{Z}_{n}(s_{j})-\mathbb{Z}_{n}(s_{j-1})).
\end{equation*}
Clearly,
\begin{eqnarray}
\left\vert \int g(s)d\widetilde{\mathbb{Z}}_{n}(s)\right\vert
\leq \left\vert \sum_{j=1}^{M}g(s_{j})(\mathbb{Z}_{n}(s_{j}^{-})-\mathbb{Z}_{n}(s_{j-1}))\right\vert 
+\left\vert \sum_{j=1}^{M}g(s_{j})(\mathbb{Z}_{n}(s_{j})-\mathbb{Z}_{n}(s_{j}^{-}))\right\vert  \label{Fidi}
\end{eqnarray}
For the first term in (\ref{Fidi}) we introduce the step function 
\begin{equation*}
g^{\ast }(t)=\sum_{j=1}^{M}1_{(s_{j-1},s_{j}]}(t)\inf_{s_{j-1}<s\leq
s_{j}}|g(s)|.
\end{equation*}
designed to approximate $|g(t)|.$ Clearly $0\leq g^{\ast }(t)\leq |g(t)|$
and 
\begin{eqnarray*}
\sum_{j=1}^{M}g^{\ast }(s_{j}) &\leq& \beta ^{-1}\sum_{j=1}^{M}g^{\ast
}(s_{j})(F_{0}(s_{j})-F_{0}(s_{j-1}))
\\
&=&\beta ^{-1}\int g^{\ast }\, {\rm d}F_{0}\\ &\leq& \beta ^{-1}\int |g|\, {\rm d}F_{0}=\beta
^{-1}||g||_{L_{1}(F_{0})}
\end{eqnarray*}
Hence, by (\ref{SupB}) we have 
\begin{eqnarray}
\left\vert \sum_{j=1}^{M}g(s_{j})(\mathbb{Z}_{n}(s_{j}^{-})-\mathbb{Z}
_{n}(s_{j-1}))\right\vert  &\leq& \Psi _{2\beta ,F_{0}}(\mathbb{Z}_{n})\left[
\sum_{j=1}^{M}(|g(s_{j})|-g^{\ast }(s_{j}))+\sum_{j=1}^{M}g^{\ast }(s_{j}))
\right]  \nonumber
\\
&\leq& \Psi _{2\beta ,F_{0}}(\mathbb{Z}_{n})(\left\Vert g\right\Vert
_{TV}+\beta ^{-1}||g||_{L_{1}(F_{0})})  \label{Comp1}
\end{eqnarray}
For the second term  in (\ref{Fidi}), we have
\begin{eqnarray}\label{vijf}
&& \left\vert \sum_{j=1}^{M}g(s_{j})\left( \mathbb{Z}_{n}(s_{j}))-\mathbb{Z}_{n}(s_{j}^{-})\right) \right\vert \nonumber\\
&&\le  \left\vert \sum_{j:\ F_0(s_j)- F_0(s_j^{-} ) \le \beta}
g(s_{j})\left( \mathbb{Z}_{n}(s_{j}))-\mathbb{Z}_{n}(s_{j}^{-})\right) \right\vert +
 \left\vert \sum_{j:\ F_0(s_j)- F_0(s_j^{-} ) > \beta}g(s_{j})\left( \mathbb{Z}_{n}(s_{j}))-\mathbb{Z}_{n}(s_{j}^{-})\right) \right\vert \nonumber\\
&&\le \Psi _{2\beta ,F_0}(\mathbb{Z}_{n}) 
\left( 2||g||_{TV}+\beta^{-1}||g||_{L_1(F_0)}\right) + {\mathbb{g} }_{\beta,F_0}( \mathbb{Z}_{n}) \beta^{-1} \| g\|_{L_1(F_0)},
\end{eqnarray}
 using for the last inequality
\[ \sum_{j=1}^M | g(s_j)| \le 2 \|g\|_{TV} +\beta^{-1}\| g\|_{L_1(F_0)}\]   and
\begin{eqnarray*}
   \sum_{j=1}^M  |g(s_{j}) | 1\{  F_0(s_j)- F_0(s_j^{-}) > \beta \} &\le &\beta^{-1} \sum_{j=1}^M | g(s_j) | |(F_0(s_j)- F_0(s_j^{-})|\\
&\le&  \beta^{-1} \| g\|_{L_1(F_0)}.\end{eqnarray*} 
Lemma 10 now follows by combining the estimates (\ref{SupB})
and (\ref{Fidi}) and (\ref{Comp1}).
\end{proof}
 
 \medskip
An immediate corollary is the following result.

\begin{corollary}
For any distribution function $F_{0}$ and for all  
$T<\infty $, $\delta >0$  and  $p\geq 1$, we have 
\begin{equation*}
\sup_{g}\left\vert \int g\,\mathrm{d}\mathbb{Z}_{n}\right\vert \leq (2\sqrt{\delta 
}+6T)\Psi _{2\sqrt{\delta },F_{0}}(\mathbb{Z}_{n})+{\mathbb{g}}_{\sqrt{ 
\delta },F_{0}}(\mathbb{Z}_{n})\sqrt{\delta },
\end{equation*}
where   the $\sup$ is taken over all right-continuous
functions $g$ with $\Vert g\Vert _{TV}\leq
T$ and $\Vert g\Vert _{L_{p}(F_{0})} = \left( \int |g|^{p}\mathrm{d}
F_{0}\right) ^{1/p}\leq \delta$.
\end{corollary}
\begin{proof}
 The proof follows trivially from Lemma 10 by
taking $\beta =\sqrt{\delta }$ and observing that $\Vert g\Vert
_{L_{1}(F_{0})}\leq \Vert g\Vert _{L_{p}(F_{0})}$ for $p\geq 1.$
\end{proof}

\subsection*{Proof of Theorem 1.}
 
First we recall, see, for instance, Chapter 1.5 in Van der Vaart and  Wellner (1996), that   $\{\bar{\mathbb{Z}}_{n}(g)$, $g\in \mathcal{G}\}$
converges weakly to a tight limit in $\ell ^{\infty }(\mathcal{G})$, provided
\begin{itemize}
\item[(a)] the marginals $(\bar{\mathbb{Z}}_n(g_1),\ldots, \bar{\mathbb{Z}}
_n(g_k))$ converge weakly for every finite subset $g_1,\ldots,g_k\in\mathcal{
G}$, and
\item[(b)] there exists a semi-metric $\rho $ on $\mathcal{G}$ such that $(
\mathcal{G},\rho )$ is totally bounded and $\bar{\mathbb{Z}}_{n}(g)$ is $
\rho $-stochastically equicontinuous, that is, 
\begin{equation*}
\lim_{\delta \downarrow 0}\limsup_{n\rightarrow \infty }\mathbb{P}\left\{
\sup_{\rho (f,g)\leq \delta }|\bar{\mathbb{Z}}_{n}(f)-\bar{\mathbb{Z}}
_{n}(g)|>\varepsilon \right\} =0
\end{equation*}
for all $\varepsilon >0$.
\end{itemize}
The finite dimensional convergence (a) follows trivially from Lemma 9,
linearity of the process $\overline{\mathbb{Z}}_{n}(g)$ and the Cram\`er-Wold
device.
As for the stochastic equicontinuity (b) of $\overline{Z}_{n}(g)$,  it
is sufficient to show that, for every $\varepsilon >0,$
\begin{equation*}
\lim_{\delta \rightarrow 0}\lim \sup_{n\rightarrow \infty }\PP\left\{
\sup_{\left\Vert h\right\Vert _{L_{1}(F_{0})}\leq \delta }\left|\int h(t)\, {\rm d}\mathbb{
Z}_{n}(t)\right|>\varepsilon \right\} =0,
\end{equation*}
where the supremum is taken over all differences $h=g-g{^{\prime }}$ with $
g,g^{\prime }\in \mathcal{G}$ and $\left\Vert h\right\Vert
_{L_{1}(F_{0})}=\int |h|dF_{0}\leq \delta .$ Since $h$ is also right-continuous and $\left\Vert h\right\Vert _{TV}\leq 2T$, Corollary 11 implies
that 
\begin{equation*}
\sup_{\left\Vert h\right\Vert _{L_{1}(F_{0})}\leq \delta }\left|\int h(t)\, {\rm d} 
\mathbb{Z}_{n}(t)\right|\leq (2\sqrt{\delta }+12T)\Psi _{2\sqrt{\delta },F_{0}}(\mathbb{Z}_{n})+{\mathbb{g}}_{\sqrt{\delta },F_{0}}(\mathbb{Z}_{n})\sqrt{\delta }.
\end{equation*}
Let $f_{t}(x)=1\{x\leq t\},$ so that $\mathbb{Z}_{n}(t)=\overline{\mathbb{Z}}_{n}(f_{t})$ and 
\begin{equation*}
d(s,t)=|F_{0}(s)-F_{0}(t)|
\end{equation*}
and observe that 
\begin{equation*}
\sup_{d(s,t)\leq \delta }|\mathbb{Z}_{n}(s)-\mathbb{Z}_{n}(t)|=\Psi _{\delta
,F_{0}}(\mathbb{Z}_{n}).
\end{equation*}%
The weak convergence of $\mathbb{Z}_{n}(t)$, or equivalently, $\overline{
\mathbb{Z}}_{n}(f_{t})$, to a continuous (with respect to $d(.,.)$) process $
\mathbb{Z}(t)$ implies that
\begin{equation*}
\Psi _{2\sqrt{\delta },F_{0}}(\mathbb{Z}_{n})\overset{P}{\rightarrow }0\text{
as }\delta \rightarrow 0\text{ and }n\rightarrow \infty .
\end{equation*}%
Moreover, the weak convergence of $\mathbb{Z}_{n}(t)$ implies that ${\mathbb{
g}}_{\sqrt{\delta },F_{0}}(\mathbb{Z}_{n})$ is bounded in probability, so ${%
\mathbb{g}}_{\sqrt{\delta },F_{0}}(\mathbb{Z}_{n})\sqrt{\delta }\rightarrow
0 $ as $\delta \rightarrow 0$ and $n\rightarrow \infty .$

\bigskip

Summarizing,   $\bar{\mathbb{Z}}_{n}(g),$   converges for each $
g\in \mathcal{G}$ to a Gaussian limit and $\bar{\ZZ}_n$ is uniformly $L_{1}(F_{0})$-equicontinuous, in probability. Moreover, $(\mathcal{G},L_{1}(F_{0}))$ is
totally bounded (for any distribution $F_{0})$. This well-known fact can be found in Example 2.6.21 in Van der Vaart and Wellner (1996, page 149).
It implies the weak convergence
of $\bar{\mathbb{Z}}_{n}(g)$ to a process with uniformly   $L_{1}(F_{0})-$continuous sample paths (see
Theorems 1.5.4 and 1.5.7 in Van der Vaart and Wellner 1996).

\bigskip

\subsection*{Proof of Theorem 6.}

Here we only prove the ``in probability" case. The almost sure statement follows after straightforward changes in the proof
of Theorem 1.  A simple modification of Lemma 9 yields%
\begin{equation*}
\int g(t)\, {\rm d}\mathbb{G}_{n}^{\ast }(t)=T_{1}(\mathbb{G}_{n}^{\ast })+T_{2}(%
\mathbb{G}_{n}^{\ast })
\end{equation*}%
for the same operators $T_{1}$ and $T_{2}$ as defined in Lemma 9. Since $%
\mathbb{G}_{n}^{\ast }$ converges to a Gaussian limit the finite dimensional
convergence follows by repeating the computation presented in the proof of
Lemma 9 after replacing $\mathbb{G}_{n}(t)$ with $\mathbb{G}_{n}^{\ast }(t)$.
As for stochastic equicontinuity of $\mathbb{Z}_{n}^{\ast }(g)$ we find,
analogous to Corollary 11, that for $\delta >0$%
\begin{equation*}
\sup_{\Vert g\Vert _{L_{p}(F_{0})}\leq \delta ,\ \Vert g\Vert _{TV}\leq
T}\left\vert \int g\,\mathrm{d}\mathbb{G}_{n}^{\ast }\right\vert \leq (2%
\sqrt{\delta }+12T)\Psi _{2\sqrt{\delta },F_{0}}(\mathbb{G}_{n}^{\ast })+{
\mathbb{g}}_{\sqrt{\delta },F_{0}}(\mathbb{G}_{n}^{\ast })\sqrt{\delta }.
\end{equation*}
Now the weak convergence of $\mathbb{Z}_{n}^{\ast }$ follows from the
convergence of $\mathbb{G}_{n}^{\ast }$, as in the proof of Theorem 1, conditionally given the sample $X_1,\ldots,X_n$.
Moreover, if $\mathbb{G}_{n}(t)$ and $\mathbb{G}_{n}^{\ast }(t)$ converge to
the same Gaussian process, then Lemma 9 coupled with the Cram\`er-Wold device
implies that the finite dimensional distribution of the limiting process of $
\mathbb{Z}_{n}(g)$ and $\mathbb{Z}_{n}^{\ast }(g)$ are the same. This
concludes the proof.

\bigskip
\appendix
\section{Appendix}

For completeness, we state the following classical result and give a simple
elementary proof which was communicated to us by David Pollard.
Theorem 18.4 in Billingsley (1986) states the result for functions on a bounded interval, yet its proof can be easily extended to the entire real line.\\

\noindent{\bf Lemma A.}
{\em
Let $f$  and $g$ be
right-continuous functions of bounded variation and define measures $\mu $ and  $\nu$ as  $\mu (-\infty ,x]=f(x)-f(-\infty )$ 
and $\nu (-\infty ,y]=g(y)-g(-\infty )$.
 Then 
\begin{equation*}
\int f(x)\,\mathrm{d}g(x)+\int g(x)\,\mathrm{d}f(x)=(fg)(\infty
)-(fg)(-\infty )+\int \int 1_{x=y}\, {\rm d} \mu(x) \,\mathrm{d}\nu(y).
\end{equation*}
Moreover, if either $f(\pm \infty )=0$ or $g(\pm \infty
)=0$, then
\begin{equation*}
\int f(x)\,\mathrm{d}g(x)+\int g(x)\,\mathrm{d}f(x)=\int \int 1_{x=y}\,
\mathrm{d}\mu \,\mathrm{d}\nu .
\end{equation*}
}
\begin{proof}
Set $H(x,y)=1\{x\leq y\}$ and observe that by the very
definition of Lebesgue integral 
\begin{eqnarray*}
f(y) &=&\int H(x,y)\, {\rm d}\mu (x)+f(-\infty )\end{eqnarray*}
and
\begin{eqnarray*}
g(x)&=&\int H(y,x)\, {\rm d}\nu (y)+g(-\infty ).
\end{eqnarray*} Hence 
\begin{eqnarray*}
&&\int f(y)\, {\rm d}\nu (y)+\int g(x)\, {\rm d}\mu (x)
\\
&&=\int \left( \int H(x,y)\, {\rm d}\mu (x)\right) \, {\rm d}\nu (y)+\int \left( \int H(y,x)\, {\rm d}\nu
(y)\right) \, {\rm d}\mu (x)\\
&&\quad+f(-\infty )(g(\infty )-g(-\infty ))+g(-\infty )(f(\infty )-f(-\infty ))
\end{eqnarray*}%
Next we apply Fubini
\begin{eqnarray*}
&&\int \left( \int H(x,y)\, {\rm d}\mu (x)\right) \, {\rm d}\nu (y)+\int \left( \int
H(y,x)\, {\rm d} \nu(y)\right) \, {\rm d}\mu (x)
\\
&&
=\int \int (H(x,y)+H(y,x))\, {\rm d}\mu (x)\, {\rm d}\nu (y)=\int \int (1_{x\leq y}+1_{y\leq
x})\, {\rm d}\mu (x)\, {\rm d}\nu (y)
\\
&&
=\int \int \, {\rm d}\mu (x)\, {\rm d}\nu (y)+\int \int 1_{x=y}\, {\rm d}\mu (x)\, {\rm d}\nu (y)
\end{eqnarray*}
to prove the lemma.
\end{proof}
 \bigskip
\subsection*{Acknowledgements.}
We would like to thank both referees for their careful reading and  many constructive remarks.\\
The research of Wegkamp was supported in part by NSF grants DMS-1310119 and DMS-1712709.\\

\end{document}